\documentclass[11pt,a4paper]{amsart}
\usepackage[all]{xy}
\usepackage{amssymb,mathrsfs,enumitem,amsmath,bbold}

\usepackage[in]{fullpage}
\usepackage{color}
\definecolor{Chocolat}{rgb}{0.36, 0.2, 0.09}
\definecolor{BleuTresFonce}{rgb}{0.215, 0.215, 0.36}
\definecolor{EgyptianBlue}{rgb}{0.06, 0.2, 0.65}

\usepackage[colorlinks,final,hyperindex]{hyperref}
\hypersetup{citecolor=BleuTresFonce, urlcolor=EgyptianBlue, linkcolor=Chocolat}

\usepackage{fourier}

\newcommand{\llbincomb}[7]{\ensuremath{\!\!\!\!
 \vcenter{\hbox{\xymatrix@R=.4pc@C=.2pc{ 
 			#4\ar@{-}[dr] && #5\ar@{-}[dl]&&\\
 			&*+[o][F-]{#3}\ar@{-}[dr]&&#6\ar@{-}[dl] \\
 			&&*+[o][F-]{#2}\ar@{-}[dr] &   &#7\ar@{-}[dl]\\
 			&&&*+[o][F-]{#1}\ar@{-}[d]&\\
 			&&&*{}&
 		}}}}}

\newcommand{\lrbincomb}[7]{\!\!\!\ensuremath{
 \vcenter{\hbox{\xymatrix@R=.4pc@C=.2pc{ 
 			&#5\ar@{-}[dr] && #6\ar@{-}[dl]&\\
 			#4\ar@{-}[dr]&&*+[o][F-]{#3}\ar@{-}[dl] \\
 			&*+[o][F-]{#2}\ar@{-}[dr] &   &#7\ar@{-}[dl]\\
 			&&*+[o][F-]{#1}\ar@{-}[d]&\\
 			&&*{}&
 		}}}}\!\!\!\!}

\newcommand{\rlbincomb}[7]{\ensuremath{
\vcenter{\hbox{\xymatrix@R=.4pc@C=.2pc{ 
			#5\ar@{-}[dr] && #6\ar@{-}[dl]&\\
			&*+[o][F-]{#3}\ar@{-}[dr]& & #7\ar@{-}[dl]\\
           #4\ar@{-}[dr] &   &*+[o][F-]{#2}\ar@{-}[dl]&\\
			&*+[o][F-]{#1}\ar@{-}[d]&&\\
			&*{}&&&
}}}}\!\!\!\!\!\!\!}
				
\newcommand{\rrbincomb}[7]{\ensuremath{
\vcenter{\hbox{\xymatrix@R=.4pc@C=.2pc{ 
			&&#6\ar@{-}[dr] && #7\ar@{-}[dl]\\
			&#5\ar@{-}[dr]& & *+[o][F-]{#3}\ar@{-}[dl]&\\
			#4\ar@{-}[dr] &   &*+[o][F-]{#2}\ar@{-}[dl]&&\\
			&*+[o][F-]{#1}\ar@{-}[d]&&&\\
			&*{}&&&&
}}}}\!\!\!\!\!\!\!\!\!\!}

\catcode`\@=11

\catcode`\@=13

\newtheorem{theorem}{Theorem}
\newtheorem{corollary}[theorem]{Corollary}

\newtheorem{lemma}[theorem]{Lemma}
\newtheorem{proposition}[theorem]{Proposition}

\newtheorem{conjecture}[theorem]{Conjecture}

\theoremstyle{definition}

\newtheorem{remark}[theorem]{Remark}
\newtheorem{definition}[theorem]{Definition}

\DeclareMathAlphabet{\pazocal}{OMS}{zplm}{m}{n}

\def\calG{\pazocal{G}}

\def\calI{\pazocal{I}}

\def\calO{\pazocal{O}}
\def\calP{\pazocal{P}}
\def\calQ{\pazocal{Q}}

\def\calS{\pazocal{S}}
\def\calT{\pazocal{T}}

\def\calX{\pazocal{X}}
\def\calY{\pazocal{Y}}

\def\ot{\otimes}
\def\sgn{{\rm sgn}}

\DeclareMathOperator{\id}{id}
\DeclareMathOperator{\Ass}{Ass}

\DeclareMathOperator{\Lie}{Lie}
\DeclareMathOperator{\Com}{Com}
\DeclareMathOperator{\Leib}{Leib}

\DeclareMathOperator{\tCom}{tCom}
\DeclareMathOperator{\pAss}{pAss}
\DeclareMathOperator{\tLie}{\widetilde{Lie}}
\DeclareMathOperator{\ttCom}{\widetilde{tCom}}
\DeclareMathOperator{\LTS}{LTS}
\DeclareMathOperator{\End}{End}

\DeclareMathOperator{\Jord}{Jord}
\DeclareMathOperator{\JTS}{JTS}

\newcommand{\di}{\mathit{di}\textnormal{-}}

\newcommand{\hadam}{\otimes}

\DeclareMathOperator{\PL}{PreLie}

\DeclareMathOperator{\Perm}{Perm}

\def\bbk{{\mathbb k}}

\newcommand{\ac}{\scriptstyle \text{\rm !`}}

\begin{document}

\title[Veronese powers]{Veronese powers of operads\\ and pure homotopy algebras}

\author{Vladimir Dotsenko}
\address{School of Mathematics, Trinity College, Dublin 2, Ireland, and Departamento de Matem\'aticas, CINVESTAV-IPN, Av. Instituto Polit\'ecnico Nacional 2508, Col. San Pedro Zacatenco, M\'exico, D.F., CP 07360, Mexico}
\email{vdots@maths.tcd.ie}

\author{Martin Markl}
\address{Mathematical Institute of the Academy, \v{Z}itn\'a 25, 115 67 Prague 1, The Czech Republic, and 
MFF UK, Sokolovsk\'a 83, 186 75 Prague 8, The Czech Republic}
\email{markl@math.cas.cz}

\author{Elisabeth Remm}
\address{Laboratoire de Math\'ematiques et Applications, Universit\'e de Haute Alsace, Facult\'e des
Sciences et Techniques, 4, rue des Fr\`eres Lumi\`ere, 68093 Mulhouse cedex, France}
\email{Elisabeth.Remm@uha.fr}


\thanks{The second author was supported by the
Eduard \v Cech Institute P201/12/G028 and RVO: 67985840.}

\begin{abstract}
We define the $m$th Veronese power of a weight graded operad $\calP$ to be its suboperad $\calP^{[m]}$
generated by operations of weight $m$.  It turns out that, unlike Veronese 
powers of associative algebras, homological
properties of operads are, in general, not improved by this construction.
However, under some technical conditions, Veronese powers of quadratic
Koszul operads are meaningful in the context of the Koszul duality theory. 
Indeed, we show that in many important cases the operads $\calP^{[m]}$ are related 
by Koszul duality to operads
describing strongly homotopy algebras with only one nontrivial
operation.
Our theory has immediate applications to objects as Lie $k$-algebras
and Lie triple systems.
In the case of Lie $k$-algebras, we also discuss a similarly looking ungraded construction 
which is frequently used in the literature. We establish that the corresponding operad does
not possess good homotopy properties, and that it leads to a very simple example of a 
non-Koszul quadratic operad for which the
Ginzburg--Kapranov power series test is inconclusive.
\end{abstract}

\maketitle

\section*{Introduction}

Many examples of algebras with $m$-ary structure operations are ``pure'' versions of homotopy algebras in the following sense. Suppose that $\calP$ is a binary quadratic Koszul operad, and $\calP_\infty=\Omega(\calP^{\ac})$ is its minimal model. We use the term \emph{pure $\calP_\infty$-algebras} for algebras over the quotient of the operad~$\calP_\infty$ by the ideal generated by all its generating operations except for those of arity $m$. (The space of generators of the thus obtained operad is homologically pure, hence the terminology.)

Another (in many ways more classical) type of algebras with $m$-ary structure operations is obtained as follows. Let $\calP$, once again, be a binary quadratic operad. We consider the suboperad of $\calP$ generated by all operations of arity $m$; experts in classical theory of identities in algebras would probably call algebras over this operad  \emph{$m$-tuple systems of type $\calP$}. (See, for instance, the work of Jacobson \cite{Jac49}, who used the term ``triple systems'' for this construction in the case of Lie and Jordan algebras.)

In this paper, we demonstrate that these two constructions are, under
some assumptions, related by Koszul duality for operads. In modern
terms, the second construction generalises the construction of
Veronese powers which is well known in the cases of graded associative
algebras (where it is known to improve homological properties).
Thus, one of the main slogans for this paper is ``Veronese
powers are Koszul dual of purifications of minimal models''.

There are many instances in the available literature where ``ungraded'' homotopy algebras are used (e.g. the homotopy Lie algebra identities are imposed for an operation of degree zero on an ungraded vector space). In our previous paper \cite{DMR16}, we already demonstrated that in the case of homotopy associative algebras it leads to extra relations and worse homotopical properties of the corresponding operad. In this paper, we prove an analogous result for the case of homotopy Lie algebras, which in particular leads to an example of a non-Koszul operad with a generator of arity three for which the inverse of the Poincar\'e series has non-negative coefficients; this is a version of a counterexample exhibited in \cite{DMR16} in the nonsymmetric case.

\subsection*{Organisation of the paper} The paper is organised as
follows. In Section \ref{sec:recoll}, we recall the key relevant
definitions of the theory of operads.  In Section
\ref{sec:VeronesePowers}, we define Veronese powers of weight graded
operads, prove some results about them, and provide counterexamples to
some celebrated properties of Veronese powers of associative
algebras. In Section \ref{sec:Identities}, we relate our work to research of
polynomial identities, both classical and recent. 
In Section \ref{sec:PureHomotopy}, we define operads 
for
pure
homotopy $\calP$-algebras, and explain how they are related to
Veronese powers by Koszul duality.
Finally, in Section~\ref{sec:mockLieInfty}, 
we demonstrate that the ungraded versions of strong
homotopy Lie
algebras sometimes used in the literature possess worse homotopical
properties than the strong homotopy Lie algebras obtained from the minimal
model of the Lie operad, and show that those algebras provide
an economic example where the positivity criterion for Poincar\'e
series does not settle the matter of non-Koszulness.

\subsection*{Acknowledgements} Some of the key results of this paper were obtained in CINVESTAV (Mexico City); the first author and the second author wish to thank that institution for the excellent working conditions. The authors are also grateful to Murray Bremner for his interest in our work (he communicated to us \cite{Bremner} that arrived at the same definition of Veronese powers independently, when trying to define what an $n$-ary algebra over an operad $\calP$ is), and to Eric Hoffbeck on comments on the first draft of the paper. 

\section{Recollections}\label{sec:recoll}

Throughout this paper, we follow the notational conventions set out in the monographs \cite{BrDo,LoVa}.  All the results of this paper are valid for vector spaces and chain complexes over an arbitrary field $\bbk$ of characteristic zero. We use the notation $\calX\cong\calY$ for isomorphisms, and the notation $\calX\simeq\calY$ for weak equivalences (quasi-isomorphisms).  

To handle suspensions of chain complexes, we introduce an element~$s$
of degree~$1$, and define, for a graded vector space~$V$, its
suspension $sV$ as $\bbk s\otimes V$. The endomorphism operad
$\End_{\bbk s}$ is denoted by $\calS$. For an operad $\calP$, its
operadic suspension is the Hadamard tensor product
$\calS\hadam\calP$. (We must warn the
  reader that in the literature the terms ``operadic suspension'' and ``operadic desuspension''
  are sometimes used in the opposite way.) 

\subsection{Three kinds of operads: weight gradings, compositions, free objects}

At various points in this paper, we use all three kinds of operads discussed in \cite{BrDo,LoVa}, that is nonsymmetric, symmetric, and shuffle operads. Our primary focus is on symmetric operads, but we have to use shuffle operads whenever operadic Gr\"obner bases are necessary or a convenient choice of a basis is required, and nonsymmetric operads for purposes of economic counterexamples. 

We require all operads in this paper to be (nonnegatively) weight
graded: this means that every component $\calP(n)$ admits a direct sum
decomposition $\calP(n)=\bigoplus_{k\ge 0}\calP(n)_{(k)}$ into
components of 
weight $k$ for various $k\ge 0$ for which any operad composition of homogeneous elements of certain weights is a homogeneous element whose weight is the sum of the weights. In addition, we assume all operads reduced ($\calP(0)=0$) and connected ($\calP(1)_{(0)}=\bbk$ and $\calP(n)_{(0)}=0$ for $n>1$). A connected operad is automatically augmented, and we denote by $\overline{\calP}$ the augmentation ideal of $\calP$.

Let us remark that each operad $\calP$ has one obvious weight grading where
 \[
\calP(n)_{(n-1)}=\calP(n) \text{ for } n \geq 1  \text{ and }
 \calP(n)_{(k)} =0 \text{ otherwise }.
 \]
For most commonly considered operads, those generated by binary operations subject to ternary relations, this grading is the most convenient one to use; in particular, it is the grading for which the relations are quadratic. However, beyond the binary generated operads, other weight gradings are occasionally more appropriate.

Each of the three kinds of operads we consider has its \emph{composition
products}. In general, for 
 \[
\alpha\in\calP(k),  \beta_1\in\calP(n_1), \ldots, \beta_k\in\calP(n_k)
 \]
and for a set partition $\pi$ of the
form $\{1,\ldots,n_1+\cdots+n_k\}=I^{(1)}\sqcup\cdots\sqcup I^{(k)}$
with $|I^{(j)}|=n_j$, 
the composition product
$\gamma_\pi(\alpha;\beta_1,\ldots,\beta_k)$ is defined. The only
difference between the three kinds of operads is in the types of
partitions permitted: in the case of nonsymmetric operads, only the
partition with
 \[
I^{(j)}=\{n_1+\cdots+n_{j-1}+1,\ldots,n_1+\cdots+n_{j}\}
 \]
is allowed,
in the case of symmetric operads, all partitions are allowed, and in
the case of shuffle operads, only the partitions with
$\min(I^{(1)})<\cdots<\min(I^{(k)})$ are allowed. 
For any kind of operads, in the case of the partition $\pi$ with
$I^{(j)}=\{n_1+\cdots+n_{j-1}+1,\ldots,n_1+\cdots+n_{j}\}$
we shall suppress $\pi$ from the subscript, as it is completely determined by the numbers
$n_j$, which in turn are completely determined by the arguments of $\gamma$: $n_j$ is the arity of $\beta_j$. 

Recall that for any kind of operads, the \emph{infinitesimal (partial) composition products} denoted $\circ_i$ are
available, such products correspond to the partitions
 \[
\{1,\ldots,n+m-1\}=\{1\}\sqcup\cdots\sqcup\{i-1\}\sqcup\{i,\ldots,i+m-1\}\sqcup\{i+m\}\sqcup\cdots\sqcup\{n+m-1\} .
 \]
In addition, for shuffle operads, we shall need \emph{infinitesimal shuffle products} $\circ_{i,\sigma}$; they correspond to the partitions 
 \[
\{1,\ldots,n+m-1\}=\{1\}\sqcup\cdots\sqcup\{i-1\}\sqcup\{i,\sigma(i+1),\ldots,\sigma(i+m-1)\}\sqcup\{\sigma(i+m)\}\sqcup\cdots\sqcup\{\sigma(n+m-1)\} ,
 \]
where $\sigma\colon\{i+1,\ldots,n+m-1\}\to\{i+1,\ldots,n+m-1\}$ is an \emph{$(m-1,n-i)$-unshuffle}, meaning that 
 \[
\sigma(i+1)<\cdots<\sigma(i+m-1) \text{ and } \sigma(i+m)<\cdots<\sigma(n+m-1) . 
 \]
Using this notion, we can define \emph{left comb products} in any shuffle operad; a left comb product of elements $x_1,\ldots, x_m$ is an element obtained from them by iterated compositions $\circ_{1,\sigma}$ where only the first slot of operations is used.

The \emph{free operad} (of each of the three kinds, where the meaning is always clarified by the surrounding context) generated by a collection
$\calX$ is denoted $\calT(\calX)$,  the \emph{cofree (conilpotent) cooperad} cogenerated by a collection $\calX$ is denoted $\calT^c(\calX)$; the
former is spanned by \emph{tree tensors}, and has its composition product given by grafting of trees, and the latter has the same underlying
collection but a different structure, a decomposition coproduct. Whenever $\calX$ is weight graded, the underlying
collection of $\calT(\calX)$ has a weight grading induced from~$\calX$. In particular, the \emph{standard} weight grading on~$\calX$
(all elements are of weight $1$) induces the \emph{standard} grading on~$\calT(\calX)$.

\subsection{Operadic Gr\"obner bases}

A very useful technical tool for dealing with operads is given by Gr\"obner bases. We refer the reader to \cite[Ch.~3, Ch.~5]{BrDo} for a systematic presentation of operadic Gr\"obner bases, and only recall the basics here. 

Similarly to associative algebras, operads can be presented via
generators and relations, that is as quotients of free operads. In
both the shuffle and the nonsymmetric case, the free operad generated
by a given nonsymmetric collection admits a basis of \emph{tree
  monomials} which can be defined combinatorially; every composition
of tree monomials is again a tree monomial. There exist several ways
to introduce a total ordering of tree monomials in such a way that the
operadic compositions are compatible with that total ordering (the
composition $\gamma_\pi$ as above, viewed as an operation with $k+1$
arguments $\alpha$, $\beta_1$, \ldots, $\beta_k$, is strictly
increasing in each argument). There is also a combinatorial definition
of divisibility of tree monomials that agrees with the naive operadic
definition: one tree monomial is a \emph{divisor} of another one if
and only if the latter can be obtained from the former by operadic
compositions. A particular case of it that we shall need is the notion of a \emph{right
  divisor}. 
A right divisor of a tree monomial $T$ is a divisor $T_1$ for which each leaf is a leaf of $T$; such a divisor really is a right divisor in that there exists a ``complementary divisor'' $T_0$ of $T$ for which $T=T_0\circ_{i,\sigma}T_1$ in the shuffle case and $T=T_0\circ_{i}T_1$ in the nonsymmetric case. 

A \emph{Gr\"obner basis} of an ideal $\calI$ of the free operad is a system $S$ of generators of~$\calI$ for which the leading tree monomial of every element of the ideal is divisible by one of the leading terms of elements of~$S$. In this case, the quotient by $\calI$ has a basis of \emph{normal tree monomials}, those not divisible by leading terms of elements of $S$. There exists an algorithmic way to compute a Gr\"obner basis starting from any given system of generators (``Buchberger's algorithm for operads''). 

Symmetric operads can be, to an extent, forced into the universe where Gr\"obner bases methods are available. For that, one uses the \emph{forgetful functor from symmetric operads to shuffle operads}. While this functor literally forgets the symmetric group actions, it does not change the underlying vector spaces, so if one wants to find a linear basis of an operad, or to prove that some vector space (of homological nature) vanishes, this is a very useful method. 

A part of the operad theory which provides one of the most useful known tools to study homological and homotopical algebra for algebras over the given operad is the Koszul duality for operads~\cite{GiKa}. A weight graded operad is said to be \emph{Koszul} if the homology of its bar complex is concentrated on the diagonal (where weight is equal to the homological degree).  If a weight graded operad is Koszul, it necessarily is quadratic, that is its defining relations are of weight two for the standard weight grading for which generators are of weight one. Proving that a given quadratic operad is Koszul instantly provides a minimal resolution for this operad, gives a description of the homology theory and, in particular, the deformation theory for algebras over that operad etc. There are a few general methods to prove that an operad is Koszul; one of the simplest and widely applicable methods is to show that a given operad has a quadratic Gr\"obner basis; this provides a sufficient (but not necessary) condition for Koszulness of an operad. In a way, non-Koszul operads can be regarded as more interesting/challenging, since standard methods of deformation theory do not work for them. 

\section{Veronese powers of operads}\label{sec:VeronesePowers}

\subsection{Na\"ive Veronese powers}

Recall the notion of Veronese powers of weight graded associative algebras: if $A=\bigoplus_{k\ge 0}A_k$, the $d$-th Veronese power of $A$, denoted by $A^{[d]}$, is the subalgebra $\bigoplus_{k\ge 0}A_{kd}$. This definition is motivated by algebraic geometry: taking the $d$-th Veronese power of the ring of polynomial functions on a vector space $V$ corresponds, under the $\mathop{\mathrm{Proj}}$ construction, to the Veronese embeddings $\mathbb{P}(V)\hookrightarrow\mathbb{P}(S^dV)$ of 
projective spaces. Veronese powers of an algebra are known to have ``better'' properties than the algebra itself, see \cite{Ba86,Ba92,EiReTo}. 

There is an obvious generalisation of this notion to the case of operads, which we call ``na\"ive Veronese powers''. 

\begin{definition}
Let $\calO$ be a weight graded operad. The \emph{na\"ive $d$-th Veronese power} $V_d(\calO)$ is the weight graded subcollection of $\calO$ defined by
 \[
(V_d(\calO))(n)=\bigoplus_{k\ge 0}\calO(n)_{(kd)}, 
 \] 
with the operad structure induced by that of $\calO$. 
\end{definition}

We proceed by justifying the adjective ``na\"ive''. Indeed, one
fundamental and easily verifiable
property that Veronese powers of associative algebras
possess is that for an algebra generated by elements of weight $1$,
its $d$-th Veronese power
is generated by elements of weight $d$. It turns out that for operads this property generally fails. 

\begin{proposition}\label{prop:NaiveVeronese}
Let $\calT=\calT(\calX)$ be the free operad generated by one commutative binary operation $\mu$. Then for all $d\ge 2$ the operad $V_d(\calT)$ is not generated by $\calT_{(d)}$. 
\end{proposition} 

\begin{proof}
  Let us denote by $\mu^{(j)}$ the iterated $j$-fold first slot
  composition of $\mu$: $\mu^{(0)}=\mathop{\mathrm{id}}$,
  $\mu^{(j+1)}=\mu^{(j)}\circ_1\mu$;
 note that
  $\mu^{(j)}\in\calT_{(j)}$. Then the element
 \[
\nu=\mu^{(d+1)}\circ_{d+1}\mu^{(d-1)}\in\calT_{(2d)}
 \]
cannot be represented as a partial composition of two elements from $\calT_{(d)}$, which is easy to see by direct inspection if we write this element as $\gamma(\mu^{(2)};\mu^{(d-1)},\mu^{(d-1)},\mathop{\mathrm{id}})$.
\end{proof}

Let us remark that in the case of operads generated by binary
operations (which is the case considered most often) the standard
weight grading is the ``arity minus one'' grading mentioned in Section~\ref{sec:recoll}, and so the na\"ive Veronese
power $V_d(\calO)$ is the suboperad on the components of $\calO$ for
all arities $n\equiv 1\pmod d$.
One can consider such suboperads in
general, but they seem to possess absolutely no notable properties
(and so are ``very na\"ive Veronese powers'').

\medskip 

Below we shall propose a more meaningful definition of Veronese
powers. As a preparation to that general definition, we first recall
the three classical notions of ``triple systems'', which correspond to
the ``Three Graces of the operad theory'' (an expression coined by Jean-Louis Loday),
the associative operad $\Ass$, the commutative associative operad $\Com$, and the operad $\Lie$ of Lie algebras. 

\subsection{Classical definitions of triple systems}

The philosophy of triple systems, going back to \cite{Jac49}, can be expressed by saying that, for a bilinear operation $a_1,a_2\mapsto a_1\star a_2$, the natural trilinear operations $a_1,a_2,a_3\mapsto (a_1\star a_2)\star a_3$ and $a_1,a_2,a_3\mapsto a_1\star (a_2\star a_3)$ are of their own merit.  (In the cases of $\Ass$, $\Com$, and $\Lie$, it is enough to take just one of those, as they can be expressed through one another using the defining relations and the symmetric group action.)

\begin{definition}[{\cite{Lis71}}]\label{def:TotallyAssocTriple}
A \emph{ternary ring}, or a \emph{totally associative triple system of
  the first type} is a vector space~$V$ with a trilinear 
operation 
 $
(-,-,-)\colon V^{\times 3}\to V
 $
satisfying the properties
\begin{equation}
((a_1,a_2,a_3),a_4,a_5)=(a_1,(a_2,a_3,a_4),a_5)=(a_1,a_2,(a_3,a_4,a_5)) .\label{eq:tass1}
\end{equation}
\end{definition}

The triple product $(a_1,a_2,a_3)=a_1a_2a_3$ in each associative algebra satisfies these identities; it is also easy to show that every multilinear identity satisfied by all those particular examples follows from \eqref{eq:tass1}.

\begin{definition}[{\cite{Lis71}}]\label{def:TotallyComTriple}
A \emph{ternary commutative ring}, or a \emph{totally commutative associative triple system} is a vector space~$V$ with a trilinear operation 
 $
(-,-,-)\colon V^{\times 3}\to V
 $
satisfying the properties
\begin{gather}
(a_1,a_2,a_3)=(a_2,a_1,a_3)=(a_1,a_3,a_2),\label{eq:tcomsym}\\
((a_1,a_2,a_3),a_4,a_5)=(a_1,(a_2,a_3,a_4),a_5)=(a_1,a_2,(a_3,a_4,a_5)) .\label{eq:tcomass}
\end{gather}
\end{definition}

The triple product $(a_1,a_2,a_3)=a_1a_2a_3$ in each commutative associative algebra satisfies these identities; it is also easy to show that every multilinear identity satisfied by all those particular examples follows from \eqref{eq:tcomsym} and \eqref{eq:tcomass}.

\begin{definition}[{\cite{Jac49,Jac68}}]\label{def:LieTriple}
A \emph{Lie triple system} is a vector space~$V$ with a trilinear
operation
 \[
[-,-,-]\colon V^{\times 3}\to V
 \]
satisfying the properties
\begin{gather}
[a_1,a_2,a_3]=-[a_2,a_1,a_3] ,\label{eq:ltsym1}\\
[a_1,a_2,a_3]+[a_2,a_3,a_1]+[a_3,a_1,a_2]=0 ,\label{eq:ltsym2}\\
[a_1,a_2,[a_3,a_4,a_5]]=[[a_1,a_2,a_3],a_4,a_5]+[a_3,[a_1,a_2,a_4],a_5]+[a_3,a_4,[a_1,a_2,a_5]] .\label{eq:ltder}
\end{gather}
\end{definition}
The triple product $[a_1,a_2,a_3]:=[[a_1,a_2],a_3]$ in any Lie algebra satisfies these identities. It is also known \cite{Jac68} that every identity satisfied by the operation $[[a_1,a_2],a_3]$ is a consequence of \eqref{eq:ltsym1}, \eqref{eq:ltsym2}, and \eqref{eq:ltder} (note that the list of identities in the paper \cite{Jac49} where Lie triple systems were first defined contains some redundant ones). We shall re-prove this result below using Gr\"obner bases for operads.  

\subsection{Veronese powers}

To ensure the property of Veronese powers which we observed to fail, the most natural way is to enforce it, which, in a sense, is exactly what the classical definition of triple systems accomplishes. 

\begin{definition}
Let $\calO$ be a weight graded operad. The \emph{$d$-th Veronese power} $\calO^{[d]}$ is the suboperad  of $\calO$ generated by $\calO_{(d)}$.
\end{definition}

To demonstrate that this definition possesses many reasonable properties, we start by computing Veronese powers of free operads. 

\begin{proposition}\label{prop:VeroneseFree} 
Let $\calT(\calX)$ be the free operad generated by $\calX$ which we
equip with the standard weight grading with $\calX$ placed in weight one.
Consider also 
the operad $\calT(\calT(\calX)_{(d)})$ which we equip with the weight grading where the weight of the generators $\calT(\calX)_{(d)}$ is equal to $d$.
We have an isomorphism of weight graded collections
 \[
\calT(\calX)^{[d]}\cong\calT(\calT(\calX)_{(d)}) .
 \] 
In other words, Veronese powers of free operads are freely generated by elements of the lowest positive~weight. 
\end{proposition}

\begin{proof}
Let us prove this for the case of shuffle operads. The case of
nonsymmetric operads is analogous (since such an operad also has a
combinatorial basis of tree monomials), and the case of symmetric
operads follows from the shuffle case via the forgetful functor.

The free shuffle operad $\calT(\calX)$ has a basis of tree monomials. By definition, the operad $\calT(\calX)^{[d]}$ is generated by $\calT(\calX)_{(d)}$. 
To prove the freeness, it is enough to show that if a~tree monomial can at all be decomposed as an iterated composition of tree monomials of weight $d$, this can be
done in an essentially unique way. 

For a tree monomial $T$, let us denote denote by $R_d(T) = \{R_1,\ldots,R_s\}$ the (possibly empty)
set of all different right divisors of $T$ with $d$ vertices. We claim that the elements in $R_d(T)$ are disjoint. 
Indeed, if $R_p$ and $R_q$ are not disjoint, they must clearly share a common leaf, say  $j$. As there is an obvious one-to-one
correspondence between the vertices of $T$ on the unique path connecting $j$ to the root and  right divisors containing $j$ as a
leaf, we conclude that $R_p=R_q$. 

Let us consider a tree monomial $T\in\calT(\calX)$. If it belongs to $\calT(\calT(\calX)_{(d)})$, the set  $R_d(T)$ is
non-empty. Indeed, if this is so, then $T$ is the result of iterated infinitesimal shuffle compositions 
 of some tree monomials $T_1,\ldots,T_u \in \calT(\calX)_{(d)}$. The tree monomial $T_u$ is a
right divisor of $T$ with $d$ vertices, i.e.\ one belonging to $R_d(T)$.    

For each tree monomial $T\in \calT(\calX)$ belonging to  $\calT(\calT(\calX)_{(d)})$, 
its underlying tree $T$ therefore determines uniquely a~nonempty set  $R_d(T) = \{R_1,\ldots,R_s\}$ 
of its (disjoint)  right divisors with $d$ vertices. We may factor out in an essentially
unique  (i.e.~modulo the axioms of operads) way the tree monomials
$R_1,\ldots,R_s$ from $T$, and induction applies. This establishes a bijection between the basis
elements of the collections $\calT(\calX)^{[d]}$ and $\calT(\calT(\calX)_{(d)})$, thus completing the proof.
\end{proof}

\begin{remark}
  The idea of the proof of Proposition~\ref{prop:VeroneseFree} offers
  an iterative test whether a~tree monomial $T\in \calT(\calX)$
  belongs to $\calT(\calT(\calX)_{(d)})$ or not. Namely, one should
  inspect the set $R_d(T)$ of its right divisors with $d$ vertices. If
  $R_d(T) = \emptyset$, then $T \not \in
  \calT(\calT(\calX)_{(d)})$. If $R_d(T)\not = \emptyset$, remove the
  trees in $R_d(T)$ from $T$ and repeat the same procedure for the
  subtree of $T$ obtained this way. For instance, for the underlying
  tree of the tree monomial $\nu$ used in the proof of
  Proposition~\ref{prop:NaiveVeronese} one easily sees that $R_d(T)=
  \emptyset$.
\end{remark}

\subsection{A criterion for Veronese powers to coincide with na\"ive ones}

As we already mentioned, the na\"ive Veronese power differs from the Veronese power as defined above in the case of the free operad.  However, in many classical cases we have $V_d(\calO)=\calO^{[d]}$, and we present one useful criterion for that.

\def\vladimirsdot{\hbox{\hskip .12em .\hskip .12em}}
\begin{proposition}\label{prop:LeftNormed}
Suppose that $\calO$ is a weight graded operad generated by operations
$\mu_1,\ldots,\mu_m$ of weight one for which every weight graded component $\calO_{(k)}$ is spanned
by symmetric groups orbits of left comb products. Then $V_d(\calO)=\calO^{[d]}$ for any $d$.
\end{proposition}

\begin{proof}
Indeed,
 \[
(\ldots (\mu_{i_1}\circ_1\mu_{i_2})\circ_1\ldots)\circ_1\mu_{i_k}=((\ldots (\mu_{i_1}\circ_1\mu_{i_2})\circ_1\ldots)\circ_1\mu_{i_d})\circ_1((\ldots (\mu_{i_{d+1}}\circ_1\mu_{i_{d+2}})\circ_1\ldots)\circ_1\mu_{i_k}) ,
 \]
so induction applies.
\end{proof}

Let us mention a class of examples defined in \cite{DoKh} for which the result we just proved allows us to relate Veronese powers of associative algebras to Veronese powers of operads. 

\begin{definition}[{\cite{DoKh}}]\label{def:ComAlgOp}
Let $A$ be a weight graded commutative associative algebra. We define a symmetric collection which we denote $\calO_A$ by putting $\calO_A(n)=A_{n-1}$ with the trivial symmetric group action, and define composition maps
 \[
\gamma\colon\calO_A(k)\otimes\calO_A(i_1)\otimes\cdots\otimes\calO_A(i_k)\to\calO_A(i_1+\cdots+i_k)
 \]
using the product
 \[
A_{k-1}\otimes A_{i_1-1}\otimes\cdots\otimes A_{i_k-1}\to A_{(k-1)+(i_1-1)+\cdots+(i_k-1)}=A_{i_1+\cdots+i_k-1} 
 \]
in the algebra $A$. 
\end{definition}

In \cite{DoKh}, it is proved that these composition maps of $\calO_A$ define an operad. 

\begin{proposition}\label{prop:VeroneseAndComAlgOp}
Suppose that $A$ is a weight graded commutative associative algebra
generated by elements of weight one. Then for each $d\ge 1$ we have
 \[
\calO_{A^{[d]}}=(\calO_A)^{[d]}=V_d(\calO_A) .
 \]
\end{proposition} 

\begin{proof}
The equality $(\calO_A)^{[d]}=V_d(\calO_A)$ follows from Proposition
\ref{prop:LeftNormed} since the first slot compositions span
$\calO_A$. The equality $V_d(\calO_A)=\calO_{A^{[d]}}$ is obvious from the definition. 
\end{proof}

\subsection{Koszulness of\texorpdfstring{\,}{} Veronese powers}

Let us demonstrate that in a large class of examples Veronese powers of operads have good homotopical properties. We start with a result that immediately connects to Proposition \ref{prop:VeroneseAndComAlgOp}.

\begin{proposition}\label{prop:HighVeroneseComAlgOp}
Suppose that $A$ is a weight graded commutative associative
algebra generated by elements of weight one. Then for all $d$ big enough 
the operad $V_d(\calO_A)$ is Koszul. 
\end{proposition}

\begin{proof}
By \cite{Ba86,EiReTo,PP}, for any weight graded commutative
associative algebra $A$, its Veronese power $A^{[d]}$ is Koszul for
all $d\gg 0$. On the other hand, by \cite{DoKh2}, the operad $\calO_B$
is Koszul whenever the algebra $B$ is Koszul, so the operad
$\calO_{A^{[d]}}$ is Koszul. 
To conclude the proof, we use the isomorphisms of Proposition~\ref{prop:VeroneseAndComAlgOp}. 
\end{proof}

Our next result produces another class of examples where Veronese powers are Koszul. It is similar to a known statement about associative algebras, see \cite[Prop. 4.4.3]{PP}.

\begin{proposition}\label{prop:PBWVeronese}
Let $\calO$ be a weight graded operad generated by elements of weight one. Assume that the Gr\"obner basis $\calG$ of the ideal of relations of $\calO$ is quadratic for a certain choice of basis in the space of generators and a certain monomial ordering, and that all weight $2d$ normal tree monomials with respect to that Gr\"obner basis are in $\calT(\calX)^{[d]}\subset\calT(\calX)$. Then the Veronese power $\calO^{[d]}$ admits a quadratic Gr\"obner basis of relations. In particular the operad $\calO^{[d]}$ is  Koszul. 
\end{proposition}

\begin{proof}
Let $\calY$ be the collection of all weight $d$ normal tree monomials with respect to $\calG$.
As the collection $\calY$ is a basis of $\calO_{(d)}$, we may choose it as the collection of generators of the operad $\calO^{[d]}$. 
There is an embedding $\calT(\calY)\hookrightarrow\calT(\calX)$ obtained via 
the embedding $\calY=\calY_{(d)}\hookrightarrow\calT(\calX)_{(d)}$, the identification
$\calT(\calX)^{[d]}\cong \calT(\calT(\calX)_{(d)}))$ 
of Proposition \ref{prop:VeroneseFree}, and the embedding $\calT(\calX)^{[d]}\hookrightarrow\calT(\calX)$.
In particular, this embedding induces a monomial ordering of $\calT(\calY)$, and this is the ordering we refer to in the rest of the proof. 

We claim that the requisite quadratic Gr\"obner basis of the operad $\calO^{[d]}$ consists of all elements in
$(\calG) \cap \calT(\calY)_{(2d)}$. To prove that, let us consider all elements $S_1,\ldots, S_t$ of $\calT(\calY)_{(2d)}$ which
are normal tree monomials with respect to $\calG$ when viewed as
elements of $\calT(\calX)_{(2d)}$ via the embedding we mentioned. By our assumption on the normal
monomials with respect to $\calG$, these elements form a basis of the weight $2d$ components of $\calO^{[d]}$. The leading terms of
$(\calG) \cap \calT(\calY)_{(2d)}$ are precisely the basis elements $\calT(\calY)_{(2d)}$ different from all of the $S_j$'s.

Thus, the result would follow if we show that a tree monomial $T\in \calT(\calY)_{(rd)}$ is \emph{not} a normal
tree monomial with respect to $\calG$ (when viewed as an element of $\calT(\calX)$) if and only if 
it has a divisor of weight $2d$ different from the $S_j$'s (when viewed as an element of $\calT(\calY)$). In $\calT(\calX)$, $T$ is not normal with respect to
$\calG$ if and only if it is divisible by one of the leading terms of $\calG$, say $T_0$. Now we recall that $T\in \calT(\calY)_{(rd)}$, and note that $T_0$ cannot be a divisor of one of the elements of $\calY$ since they all are assumed normal. Also, since the weight of $T_0$ is $2$, so its occurrence overlaps with exactly two of the elements of $\calY$, which determines a divisor of weight $2d$ different from one of the $S_j$'s, thus completing the proof.
\end{proof}

\subsection{Non-properties of Veronese powers of operads. }
 
We conclude this section by showing that most classically known stronger results valid for Veronese powers of associative algebras are not true for operads in general.

\begin{proposition}\label{prop:non-examples} The following results valid for Veronese powers of algebras are, in general, false in the case of operads:
\begin{enumerate}
\item ``improvement of defining relations for Veronese powers''
  (\cite[Prop.~3]{BaFr85}): if all elements of the minimal set of
  defining relations of $A$ are of weight at most $(k-1)d+1$, then all
  elements of the minimal set of defining relations of $A^{[d]}$ are
  of weight at most $k$ (for the standard grading of the Veronese power where all generators have weight one);
\item ``improvement of the slope of the off-diagonal homology of the bar complexes for Veronese powers'' (\cite[Th.~1]{Ba86}): the rate $r(A)$ defined as $\sup\left\{\frac{j-1}{i-1} \colon (\mathop{\mathrm{Tor}}_i^A)_{(j)}(\bbk,\bbk)\ne 0 \right\}$ satisfies $r(A^{[d]})\le \left\lceil \frac{r(A)}{d}\right\rceil$;
\item ``high'' Veronese powers of a weight graded algebra $A$ with a finite Gr\"obner basis admit a quadratic Gr\"obner basis (\cite[Sec.~4.4, Remark]{PP});
\item all Veronese powers of a weight graded algebra $A$ with a quadratic Gr\"obner basis admit a quadratic Gr\"obner basis (\cite[Prop.~4.4.3]{PP}).
\end{enumerate}
\end{proposition}

\begin{proof}
We shall exhibit two examples which give a good idea of what is going on.

\textit{Example 1. } Consider the nonsymmetric operad $\calO$ with one
binary generator $\nu$ subject to the only monomial relation
$\omega\circ_2(\omega\circ_2\omega)=0$ of weight $3$. Then for every $d\ge 1$
the minimal set of relations presenting the Veronese power
$\calO^{[d]}$ as a quotient of $\calT(\calO_{(d)})$ has weight $3$. To
demonstrate this, note that, in the notation of the proof of
Proposition \ref{prop:NaiveVeronese},
in the operad $\calO$ we have
$\omega^{(d)}\ne 0$ and $\omega^{(d)}\circ_{d+1}\omega^{(d)}\ne 0$, but however
\hbox{$\omega^{(d)}\circ_{d+1}(\omega^{(d)}\circ_{d+1}\omega^{(d)})=0$}, which easily
implies that the latter element belongs to the minimal set of
relations of $\calO^{[d]}$. As monomial relations always form a
Gr\"obner basis, this operad clearly serves as a counterexample to
statements (1) and (3).

\textit{Example 2. } Consider the nonsymmetric operad $\calO$ with three binary generators $\mu$, $\nu$, and $\rho$ subject to relations
\begin{gather*}
\mu\circ_1\nu=\rho\circ_2\nu, \quad  \rho\circ_1\nu=0,\\ 
\mu\circ_1\mu=\mu\circ_2\mu=\mu\circ_1\rho=\mu\circ_2\rho=\rho\circ_1\rho=\rho\circ_2\rho=0,\\
\rho\circ_1\mu=\rho\circ_2\mu=\nu\circ_1\mu=\nu\circ_2\mu=\nu\circ_1\rho=\nu\circ_2\rho=0.
\end{gather*}
For the ordering $\mu>\rho>\nu$, these relations form a Gr\"obner
basis of relations of $\calO$ (this is easy to check directly: most
S-polynomials are \emph{equal} to zero, not just have zero as a normal
form). In particular, this operad is Koszul, and $r(\calO)=1$, as the
nonzero homology of the bar construction
is concentrated on the diagonal.

Note that the normal monomials of weight two in $\calO$ are $A=\nu\circ_1\nu$, $B=\mu\circ_2\nu$, $C=\rho\circ_2\nu$, and $D=\nu\circ_2\nu$. Most of the compositions of these elements  are either zero or normal monomials of weight four with respect to the Gr\"obner basis of $\calO$; more precisely, by a direct computation, we have
\begin{gather*}
A\circ_1B=A\circ_2B=A\circ_3B=B\circ_1B=B\circ_2B=B\circ_3B=0,\\
C\circ_1B=C\circ_2B=C\circ_3B=D\circ_1B=D\circ_2B=D\circ_3B=0,\\
A\circ_1C=A\circ_2C=A\circ_3C=B\circ_1C=B\circ_2C=B\circ_3C=0,\\
C\circ_1C=C\circ_2C=C\circ_3C=D\circ_1C=D\circ_2C=D\circ_3C=0,\\
B\circ_1A=C\circ_1A=C\circ_1D=0
\end{gather*}
and all other weight two compositions of $A$, $B$, $C$, and $D$ are normal monomials with respect to the Gr\"obner basis of $\calO$ with one exception of the element $B\circ_1 D$. This latter element is not a normal monomial with respect to the Gr\"obner basis of $\calO$, and its normal form with respect to that Gr\"obner basis is the element $(\rho\circ_2(\nu\circ_2\nu))\circ_2\nu$, which does not belong to $\calT(\mu,\nu,\rho)^{[2]}$. It follows that the basis of $\calO^{[2]}_{(4)}$ is given by the composites of $A$, $B$, $C$, and $D$ that are normal monomials with respect to the Gr\"obner basis of $\calO$, and the element $B\circ_1 D$. Examining elements of weight $6$, we discover that there are two relations between $A$, $B$, $C$, and $D$ that do not follow from the weight four ones:
 \[
B\circ_1(D\circ_1A)=0=B\circ_1(D\circ_1D) . 
 \] 
For instance, 
 \[
(B\circ_1 D)\circ_1 A=((\rho\circ_2(\nu\circ_2\nu))\circ_2\nu)\circ_1 (\nu\circ_1\nu)=((\rho\circ_1\nu)\circ_1\nu)\circ_2((\nu\circ_2\nu)\circ_1\nu) ,
 \]
and it remains to recall that $\rho\circ_1\nu=0$. We conclude that the minimal set of relations of the operad~$\calO^{[2]}$ is not quadratic, so the bar complex of this operad operad has off-diagonal homology classes and this operad is not Koszul, so its rate of growth of homology is larger than $1$ (the rate of any Koszul operad). Thus this operad serves as a counterexample to statements (2) and (4). 
\end{proof}

\section{Research of polynomial identities from the Veronese viewpoint}\label{sec:Identities}

In this section, we discuss some work on polynomial identities, both classical and recent, in the language of Veronese powers. 

\subsection{Na\"ive Veronese powers and classical triple systems}

Proposition \ref{prop:LeftNormed} (or its version where composition in the first slot is replaced by composition in the last slot) is applicable, among other case, to the ``Three Graces'' $\Ass$, $\Com$, $\Lie$. For $d=2$, this recovers the classical definitions of triple systems.

\begin{proposition}\label{prop:ClassicalTS}\leavevmode
\begin{itemize}
\item $V_2(\Ass)=\Ass^{[2]}$ is the operad controlling totally associative triple systems;
\item $V_2(\Com)=\Com^{[2]}$ is the operad controlling totally commutative associative triple systems; 
\item $V_2(\Lie)=\Lie^{[2]}$ is the operad $\LTS$ controlling Lie triple systems.
\end{itemize}
\end{proposition}

\begin{proof}
In each of these statements, the equality follows from Proposition \ref{prop:LeftNormed}, and it only remains to check that there are no other relations. 
In the first two cases, this is obvious, as already the relations we have give a tight upper bound on the size of the Veronese power. For the third one, to show that all the defining relations of $V_2(\Lie)=\Lie^{[2]}$ are quadratic, we use Proposition \ref{prop:PBWVeronese}. Indeed, it is known \cite{DoKh} that for the reverse path degree-lexicographic order on shuffle trees the defining relation of $\Lie$ forms a Gr\"obner basis, so the left comb products of the generator are precisely the normal tree monomials. Thus, the operad $\Lie^{[2]}$ is Koszul and therefore quadratic. 
\end{proof}

\subsection{Beyond the na\"ive Veronese powers}

Most of our work focussed on operads covered by Proposition \ref{prop:LeftNormed}, which are operads for which Veronese powers coincide with na\"ive Veronese powers. We shall now discuss two important cases going beyond this framework, a very classical one in the context of polynomial identities, and a very recent one. 

\smallskip 

One important example not covered by Proposition \ref{prop:LeftNormed}
is the operad of Jordan algebras. Those algebras are usually defined
as commutative (non-associative) algebras
satisfying the identity 
\[
(ab)(aa) = a(b(aa))
\]
which is not multilinear. Since we work in the
context of operad theory, we recall an equivalent multilinear version.

\begin{definition}
The operad $\Jord$ of Jordan algebras is generated by one operation $a_1,a_2\mapsto a_1a_2$, subject to the relation 
 \[
((a_1a_2)a_3)a_4+((a_1a_4)a_3)a_2+((a_2a_4)a_3)a_1=(a_1a_2)(a_3a_4)+(a_1a_3)(a_2a_4)+(a_1a_4)(a_2a_3) .
 \] 
\end{definition}

The classical definition of a Jordan triple system \cite{Jac49} as
recalled below 
is intimately related to our notion of Veronese powers, although the way that definition is usually given slightly obscures this~fact. 

\begin{definition}
A \emph{Jordan triple system} is a vector space $V$ with a trilinear operation $\{-,-,-\}\colon V^{\times 3}\to V$ satisfying the properties
\begin{gather*}
\{a_1,a_2,a_3\}=\{a_3,a_2,a_1\}, \label{eq:jtsym} \\
\{a_1,a_2,\{a_3,a_4,a_5\}\}=\{\{a_1,a_2,a_3\},a_4,a_5\}-\{a_3,\{a_2,a_1,a_4\},a_5\}+\{a_3,a_4,\{a_1,a_2,a_5\}\} . \label{eq:jtder}
\end{gather*}
\end{definition} 

These identities are satisfied by the operation $\{a_1,a_2,a_3\}=(a_1a_2)a_3+a_1(a_2a_3)-a_2(a_1a_3)$ in any Jordan algebra. One can easily check that this operation generates $\Jord_{(2)}$ as an $S_3$-module, thus the suboperad of $\Jord$ generated by this operation is $\Jord^{[2]}$. An easy computation shows that the relations of the operad $\JTS$ of Jordan triple systems are the only \emph{quadratic} relations satisfied by this operation, so in particular,  we have a natural surjection from the operad $\JTS$ to the operad $\Jord^{[2]}$.  To the best of our knowledge, it is not known whether $\JTS\cong\Jord^{[2]}$. We conjecture that it is the case.

\begin{conjecture}\label{conj:JordVeron}
We have  $\JTS\cong\Jord^{[2]}$. In other words, all relations
satisfied by the operation 
\[
(a_1a_2)a_3+a_1(a_2a_3)-a_2(a_1a_3)
\] 
in any Jordan algebra follow from the axioms of Jordan triple systems. In yet other words, the operad $\Jord^{[2]}$ is quadratic.
\end{conjecture}

Of course, it is even less clear what the higher Veronese powers $\Jord^{[k]}$ are. Note that the ``Jordan quadruple systems'' \cite{BrMadQuad} are defined in a way that it is not at all clear how they are related to Veronese powers.

\smallskip 

Another important example of an operad not covered by Proposition \ref{prop:LeftNormed}  is the pre-Lie operad.

\begin{definition}
The operad $\PL$ of pre-Lie algebras is generated by one operation $a_1,a_2\mapsto a_1a_2$, subject to the relation 
 \[
(a_1a_2)a_3-a_1(a_2a_3)=(a_1a_3)a_2-a_1(a_3a_2) .
 \]
\end{definition}

The Veronese square of this operad was essentially studied by Bremner and Madariaga \cite{BrMadPL}; their computations take place inside the free dendriform algebra, but as the operad $\PL$ is well known to be a suboperad of the dendriform operad, this is sufficient to study free pre-Lie algebras and the pre-Lie operad. 

\begin{proposition}[{\cite[Th.~4.5]{BrMadPL}}]
Consider the following two elements of the operad $\PL^{[2]}$: 
 \[
[a_1,a_2,a_3]_1=(a_1a_2)a_3, \quad  [a_1,a_2,a_3]_2=a_1(a_2a_3) .
 \]
All relations of weight at most two between these elements follow from the relations
 \[
[a_1, a_2, a_3]_1 - [a_1, a_2, a_3]_2 =  [a_1, a_3, a_2]_1 - [a_1, a_3, a_2]_2 , 
 \]
\begin{gather*}
[[a_1, a_2, a_3]_1, a_4, a_5]_2 - [[a_1, a_4, a_5]_2, a_2, a_3]_1 + [a_1, [a_4, a_5, a_2]_1, a_3]_1 - [a_1, [a_2, a_4, a_5]_2, a_3]_1\\
+ [a_1, a_2, [a_4, a_5, a_3]_1]_1 - [a_1, a_2, [a_3, a_4, a_5]_2]_1 =0 ,\\
[[a_1, a_2, a_3]_1, a_4, a_5]_1 - [[a_1, a_4, a_2]_1, a_3, a_5]_1 + [[a_1, a_4, a_3]_1, a_2, a_5]_1 - [[a_1, a_3, a_2]_1, a_4, a_5]_1\\
- [a_1, [a_2, a_3, a_4]_1, a_5]_1 + [a_1, [a_4, a_2, a_3]_1, a_5]_1 - [a_1, [a_4, a_3, a_2]_1, a_5]_1 + [a_1, [a_3, a_2, a_4]_1, a_5]_1 =0,\\
[[a_1, a_2, a_3]_1, a_4, a_5]_2 - [[a_1, a_4, a_5]_2, a_2, a_3]_1 + [[a_1, a_4, a_5]_2, a_2, a_3]_2 - [[a_1, a_2, a_3]_2, a_4, a_5]_2\\
+ [a_1, [a_4, a_5, a_2]_1, a_3]_1 + [a_1, [a_4, a_5, a_3]_1, a_2]_1 - [a_1, [a_2, a_3, a_5]_1, a_4]_1 - [a_1, [a_4, a_5, a_3]_1, a_2]_2\\
+ [a_1, [a_2, a_3, a_5]_1, a_4]_2 - [a_1, [a_2, a_4, a_5]_2, a_3]_1 - [a_1, [a_3, a_4, a_5]_2, a_2]_1 - [a_1, [a_4, a_2, a_3]_2, a_5]_2\\
+ [a_1, [a_2, a_4, a_5]_2, a_3]_2 + [a_1, [a_3, a_4, a_5]_2, a_2]_2 + [a_1, a_4, [a_2, a_3, a_5]_1]_1 - [a_1, a_4, [a_5, a_2, a_3]_2]_2 = 0.
\end{gather*}
\end{proposition}

Bremner and Madariaga prove \cite[Th.~4.5]{BrMadPL} that the operad $\PL^{[2]}$ has no cubic relations that do not follow from the quadratic ones, and make a conjecture which may be re-stated in terms of Veronese powers as follows.

\begin{conjecture}[{\cite[Conj.~4.8]{BrMadPL}}]
The operad $\PL^{[2]}$ is quadratic.
\end{conjecture}

\subsection{Dialgebras and Veronese powers}

In some recent papers on polynomial identities, the notion of $\calP$-dialgebras (for various operads $\calP$) has been studied in the context of Veronese powers. Let us recall the relevant definition and its historical context. Recall that the operad $\Perm$ of ``permutative'' algebras encodes (non-unital) associative algebras satisfying the extra identity $abc=acb$. 

\begin{definition}
Let $\calP$ be an operad. The operad $\di\calP$ is the Hadamard product $\calP\hadam\Perm$. Algebras over this operad are called \emph{\/$\calP$-dialgebras}.
\end{definition}

This definition goes back to Chapoton \cite{Chap}. Later, explicit algorithms for computing all identities of $\calP$-dialgebras from all identities of $\calP$-algebras were suggested by several authors \cite{BSO,Kol,Pozh}, and these algorithms were proved to give the same result~\cite{KV}. 

\smallskip

Let us establish a result on Veronese powers of Hadamard products of operads which we shall then use to relate the definition of $\calP$-dialgebras to Veronese powers. 
Let $\calO$ be an operad and \[\mu^\calO\colon \calT(\calO) \to \calO\]  the
natural map given by the structure operations of~$\calO$. Recall that
the arity $n$ component
$\calT(\calO)(n)$ of the free  operad  $\calT(\calO)$ decomposes into the direct sum,
\[
\calT(\calO)(n) = \bigoplus_{[T]} \calT_T(\calO),
\]
over isomorphism classes of labelled rooted trees with $n$
leaves. Denote finally by 
\begin{equation}\label{eq:rest}
\mu_T^\calO :
\calT_T(\calO) \to \calO(n),
\end{equation}
the restriction of $\mu^\calO$ to $\calT_T(\calO)$.

\begin{proposition}
Let $\calP$ be a weight graded operad and $\calQ$ an operad. Let us
equip the Hadamard product $\calP \hadam \calQ$ with a weight grading
by postulating that the weight of $p \ot q \in \calP \hadam \calQ$
equals the weight of $p$. Then 
\[
V_k(\calP \hadam \calQ) = V_k(\calP) \hadam \calQ\
 \hbox { and } \ (\calP \hadam \calQ)^{[k]} \subset \calP^{[k]}
\hadam \calQ.
\]
Suppose moreover that all restrictions $\mu^\calQ_T$ as in~\eqref{eq:rest} are
epimorphisms. Then the inclusion above is actually an equality
\[
(\calP \hadam \calQ)^{[k]} =  \calP^{[k]} \hadam \calQ.
\] 
\end{proposition}

\begin{proof}
The first part of the proposition is elementary.  To prove the second
one, notice that elements of $\calP^{[k]} \hadam \calQ$ are linear
combinations of elements of the form 
$
\mu^\calP_T(x) \ot q, 
$
where $x \in \calT_T(\calP_{(k)})$ and $q \in \calQ$. By assumption, $q =
\mu^\calQ_T(y)$ for some $y \in \calT_T(\calQ)$, therefore 
\[
\mu^\calP_T(x) \ot q = \mu^\calP_T(x) \otimes \mu^\calQ_T(y) =
\mu^{\calP \hadam \calQ}_T (x \ot y) \in (\calP \hadam \calQ)^{[k]} ,
\] 
as claimed.
\end{proof}

The assumption required by the second part of the proposition means,
in plain language, that given an arbitrary composition scheme, every
element of $\calQ$ is decomposable with respect to this scheme. If 
the operad $\calQ$ is generated by binary operations, it is sufficient to 
require that the restrictions $\mu^\calQ_T$ are epimorphisms for all binary 
trees $T$. Such operads were studied in \cite{Val08}, where it was proved
that for such operad $\calQ$, the Hadamard product $\calP\hadam\calQ$ 
coincides with the Manin white product of $\calP$ and $\calQ$. In particular,
this property holds for the operad $\Perm$, see \cite[Cor.~16]{Val08}. We therefore have

\begin{corollary}\label{cor:VeroneseDialg}
Let $\calP$ be a weight graded operad. For such an operad, 
the  operad of dialgebras construction commutes with Veronese powers: 
 \[ 
V_k(\di\calP) =\di V_k(\calP) \quad \text{and} \quad  (\di\calP)^{[k]} = \di (\calP^{[k]}) .
 \]
\end{corollary}

\medskip

In \cite{BrSO}, Leibniz triple systems are introduced, taking as a starting point the operad $\LTS$ of Lie triple systems. Namely, Leibniz triple systems are defined as  $\LTS$-dialgebras. One of the results of \cite{BrSO} can be re-stated in the language of our paper as follows. 

\begin{proposition}[{\cite[Th.~23]{BrSO}}]
The Veronese power $\Leib^{[2]}$ is isomorphic to the operad of Leibniz triple systems.
\end{proposition}

Let us give a short alternative proof of this result. By \cite{Chap}, $\Leib=\di\Lie$, and by Proposition \ref{prop:ClassicalTS}, $\LTS\cong\Lie^{[2]}$. Finally, by Corollary \ref{cor:VeroneseDialg},
 \[
\Leib^{[2]}=(\di\Lie)^{[2]}=\di(\Lie^{[2]})\cong\di\LTS ,
 \]
as required.

\medskip

In \cite{BFSO}, the authors define and study a version of the previous definition
where Lie algebras are replaced by Jordan algebras. Namely, they define a Jordan triple disystem 
as a $\JTS$-dialgebra. One of the results
proved in \cite{BFSO} can be re-stated in the language of our paper as
follows. 

\begin{proposition}[{\cite[Th.~7.3]{BFSO}}]
Generating operations of the operad $(\di\Jord)^{[2]}$ satisfy the identities of Jordan triple disystems.  
\end{proposition}

Let us give a short alternative proof of this result. Indeed, by Corollary \ref{cor:VeroneseDialg}, we have 
 \[
(\di\Jord)^{[2]}=\di(\Jord^{[2]}) ,
 \]
and $\Jord^{[2]}$ is a quotient of $\JTS$, so $\di(\Jord^{[2]})$ is a quotient of $\di\JTS$, as required.

\section{Operads of pure homotopy algebras and Koszul duality}\label{sec:PureHomotopy}

In this section, we explain how to compute Koszul duals of Veronese powers of quadratic operads. Let us begin with stating a result on Koszul duals of operads $\Ass$ and $\Com$ which can be proved by a direct calculation. 

\begin{proposition} \leavevmode
\begin{itemize}
\item The Koszul dual operad of $\Ass^{[k]}$ is the operad controlling
  a particular class of $A_\infty$-algebras, those with only nonzero
  structure operation being the one of arity $k+1$, or, equivalently,
  the operad of partially associative $(k+1)$-ary algebras with the
  structure operation of homological degree $k-1$.
\item The Koszul dual operad of $\Com^{[k]}$ is the operad controlling a particular class of $L_\infty$-algebras, those with only nonzero structure operation being the one of arity $k+1$, or, equivalently, the operadic desuspension of the the operad of Lie $(k+1)$-algebras \cite{HaWa}.
\end{itemize}
\end{proposition}

The fact that $\big(\Ass^{[k]}\big)^!$ has to do 
with $A_\infty$-algebras and $\big(\Com^{[k]}\big)^!$ has to do with
$L_\infty$-algebras makes one expect that a general statement relating
Veronese powers to homotopy algebras exists. The goal of this section
is to confirm this guess. For that, we introduce a new general notion of pure
homotopy $\calP$-algebras, and then prove that (appropriate quadratic versions of) Veronese powers of operads 
and operads controlling pure homotopy algebras are exchanged by Koszul duality. 

\begin{definition} Let $k>0$ be an integer.
\begin{itemize}
\item Let $\calQ$ be a connected weight graded cooperad. The \emph{weight $k$ Koszul dual operad} $\calQ^{\ac}_{[k]}$ is the largest quotient dg operad of the cobar complex $\Omega(\calQ)=(\calT(s^{-1}\overline{\calQ}),d_{\Omega(\calQ)})$ generated (as a non dg operad) by $s^{-1}\overline{\calQ}_{(k)}$. 
\item Let $\calP$ be a connected weight graded operad. The
  \emph{weight $k$ Koszul dual cooperad} $\calP^{\ac}_{[k]}$ is the
  smallest dg subcooperad of the bar complex
  $\mathsf{B}(\calP)=(\calT(s\overline{\calP}),d_{\mathsf{B}(\calP)})$
 cogenerated (as a non dg cooperad) by $s\overline{\calP}_{(k)}$.
\end{itemize}
\end{definition}

The following proposition makes this definition more explicit. 

\begin{proposition}\label{prop:Pure}\leavevmode
\begin{itemize}
\item For a connected weight graded cooperad $\calQ$ and an integer $k>0$, the operad $\calQ^{\ac}_{[k]}$ is the (non dg) operad that can be presented via generators  $s^{-1}\overline{\calQ}_{(k)}$ and relations which are desuspended quadratic co-relations on $\overline{\calQ}_{(k)}$ in $\calQ$.
\item For a connected weight graded operad $\calP$ and an integer $k>0$, the cooperad $\calP^{\ac}_{[k]}$ is the (non dg) cooperad that can be presented via co-generators $s\overline{\calP}_{(k)}$ and co-relations which are suspended quadratic relations on $\overline{\calP}_{(k)}$ in $\calP$.
\end{itemize}
  \end{proposition}

\begin{proof}
To prove the first statement, we note that the underlying operad of
$\Omega(\calQ)$ is free, so to create a quotient generated by
$s^{-1}\overline{\calQ}_{(k)}$, we have to quotient out all elements
$s^{-1}\overline{\calQ}_{(l)}$, $l\ne k$. Furthermore, any dg ideal of
$\Omega(\calQ)$ containing all elements of $s^{-1}\overline{\calQ}$ of
weight different from $k$ must contain all the differentials of these
elements, and so the smallest such ideal coincides with the (usual)
ideal of $\calT(s^{-1}\overline{\calQ})$ generated by all elements of
$s\overline{\calQ}$ of weight different from $k$ and their
differentials. Note that the differential of any homogeneous element
of weight different from $2k$ is a combination of tree tensors
involving at least one element of weight different from $k$, so modulo
the ideal generated by elements of weight different from $k$ such an
element is congruent to zero. Also, modulo the same ideal the
differential of any homogeneous element of weight $2k$ is congruent to
the suspended image of the decomposition map $\calQ_{(2k)}\to \calT^c(\calQ_{(k)})_{(2)}$. 
These observations establish both the shape of relations and the lack of differential in the quotient. The second statement is proved analogously. 
\end{proof}

Let us give some examples of operads of pure homotopy algebras which agree with our previous observations.

\begin{proposition}\leavevmode
\begin{itemize}
\item Let $\calQ=\Com^c$ be the linear dual of\/ $\Com$. Then the
  operad $\calQ^{\ac}_{[k]}$ is the operadic suspension 
of the operad of Lie $(k+1)$-algebras of~\cite{HaWa}. 
\item Let $\calQ=\Ass^c$ be the linear dual of\/ $\Ass$. Then the operad $\calQ^{\ac}_{[k]}$ is the operad $\widetilde{\pAss^{k+1}_{-1}}$ of~\cite{MaRe15}. 
\item Suppose that $\calQ$ is quadratic cooperad (cogenerated by elements of weight one). Then the operad
  $\calQ^{\ac}_{[1]}$ is the classical quadratic dual operad
  $\calQ^{\ac}$. Similarly,  if\/ $\calP$ is a quadratic operad (generated by elements of weight one), then the
  cooperad $\calP^{\ac}_{[1]}$ 
is the classical quadratic dual cooperad~$\calP^{\ac}$.
\end{itemize}
\end{proposition}

\begin{proof}
These are obtained from Proposition \ref{prop:Pure} by a direct inspection.
\end{proof}

In general, Veronese powers need not be quadratic, so in order to include Veronese powers in the context of Koszul duality, we should modify the definition appropriately.

\begin{definition}\label{def:qVeronese}
Let $\calP$ be a weight graded operad, and let $d>0$ be an integer. The \emph{quadratic Veronese power $q\calP^{[d]}$} is the quotient of $\calT(\calP_{(d)})$ by the ideal of all quadratic relations satisfied by $\calP_{(d)}$, that is $\calI\cap\calT(\calP_{(d)})_{(2)}$, where $\calI\subset \calT(\calP_{(d)})$ is the kernel of the evaluation map $\calT(\calP_{(d)})\to\calP$.
\end{definition}

In general, there is a surjection $q\calP^{[d]}\twoheadrightarrow\calP^{[d]}$, and an inclusion $\calP^{[d]}\hookrightarrow V_d(\calP)$, so the set-up we are dealing with is reminiscent of the situation with Manin products and Hadamard products, see \cite{Val08}.

\medskip 

We are ready to formulate the main result of this section. 

\begin{theorem}\label{th:VeronesePureKoszul}
Let $\calP$ be a connected weight graded operad, 
and let $d>0$ be an integer. We have 
 \[
\left(q\calP^{[d]}\right)^{\ac}\cong \calP^{\ac}_{[d]} .
 \]
\end{theorem}

Note that this statement is valid in full generality, that is the operad $\calP$ need not be quadratic, its Veronese power need not be quadratic etc. However, it becomes particularly meaningful under some assumptions on $\calP$, for example, assuming that $\calP$ is quadratic and that $q\calP^{[d]}=\calP^{[d]}$. 

\begin{proof}
The result is stated in such a way that the proof becomes almost tautological. Indeed, the cooperad $\left(q\calP^{[d]}\right)^{\ac}$ is the subcooperad of 
$\calT^c(s\calP_{(d)})$ presented by co-relations which are suspended quadratic relations of $\calP^{[d]}$. By Proposition \ref{prop:Pure}, this cooperad is isomorphic to $\calP^{\ac}_{[d]}$.
\end{proof}

The following proposition demonstrates how our results can be applied in particular cases to produce results which are far from obvious. 

\begin{proposition}\label{prop:LieTriple}\leavevmode
\begin{itemize}
\item The operad $\mathop{\mathrm{LTS}}$ encoding Lie triple systems is Koszul. Its Koszul dual is the operad $\Com_{\infty,3}$ encoding $C_\infty$-algebras whose nonzero structure operations are in arity $3$. 
\item More generally, for every $k\ge 1$ we have $V_k(\Lie)=\Lie^{[k]}=q\Lie^{[k]}$. This operad is Koszul, and its Koszul dual is the operad $\Com_{\infty,k+1}$ encoding $C_\infty$-algebras whose nonzero structure operations are in arity $k+1$. 
\end{itemize}
\end{proposition}

\begin{proof}
The first statement is the particular case $k=2$ of the second one. To prove the second one, we note that from the argument of Proposition \ref{prop:ClassicalTS} it follows that the operad $\Lie^{[k]}$ is Koszul, and in particular quadratic, so  $\Lie^{[k]}=q\Lie^{[k]}$. Also, we already know from Proposition \ref{prop:LeftNormed} that
$V_k(\Lie)=\Lie^{[k]}$. Finally, Theorem~\ref{th:VeronesePureKoszul}
 implies that $\left(q\Lie^{[k]}\right)^{\ac}\cong \Lie^{\ac}_{[k]}$. The suspended dual of the latter is manifestly $\Com_{\infty,k+1}$, which completes the proof.   
\end{proof}

We conclude this section with an amusing corollary showing how Bernoulli numbers arise in dimension formulas for pure $\Com_{\infty}$-algebras. For the case of pure $\Lie_\infty$-algebras, we refer the reader to \cite[Example~4.3.5]{Wa07} (quite remarkably, in that case there is a simple closed dimension formula $\dim\Lie_{\infty,3}(2n-1)=((2n-3)!!)^2$).  

\begin{corollary}
For each $n\ge 1$, we have 
 \[
\dim\Com_{\infty,3}(2n-1)=\frac{2^{2n}(2^{2n}-1)|B_{2n}|}{2n},
 \]
where $B_{2n}$ is the $2n$-th Bernoulli number. 
\end{corollary}

\begin{proof}
By Proposition~\ref{prop:LieTriple} the operad of Lie triple systems is Koszul, and its Koszul dual cooperad is $\Lie^{\ac}_{[2]}$. 
That cooperad is cogenerated by ternary operations of homological degree $1$, thus, its component of arity $2n-1$ 
is concentrated in homological degree $n-1$. The Ginzburg--Kapranov functional equation relating the Poincar\'e series  
of an operad and its Koszul dual \cite{GiKa} shows that the generating series of dimensions of the operad $\LTS$ 
is the compositional inverse of the series
 \[
\sum_{n\ge 1}\frac{(-1)^{n-1}\dim\Com_{\infty,3}(2n-1)}{(2n-1)!}t^{2n-1} . 
 \]
Equivalently, the series
 \[
g(t)=\sum_{n\ge 1}\frac{(-1)^{n-1}\dim\LTS(2n-1)}{(2n-1)!}t^{2n-1}  
 \]
with modified signs is the compositional inverse of the series
 \[
\sum_{n\ge 1}\frac{\dim\Com_{\infty,3}(2n-1)}{(2n-1)!}t^{2n-1} .
 \]
For the operad $\LTS$ of Lie triple systems, we have $\LTS\cong V_2(\Lie)$, so $\dim\LTS(2n-1)=\dim\Lie(2n-1)=(2n-2)!$ for each $n\ge 1$. Thus, 
 \[
g(t)=\sum_{n\ge 1}\frac{(-1)^{n-1}(2n-2)!}{(2n-1)!}t^{2n-1}=\sum_{n\ge 1}\frac{(-1)^{n-1}t^{2n-1}}{2n-1}=\arctan(t) ,
 \]
so from the well known formula \cite[Appendix~B]{MiSta}
 \[
\tan(t)=\sum_{n\ge 1}\frac{2^{2n}(2^{2n}-1)|B_{2n}|}{(2n)!}t^{2n-1} 
 \]
it immediately follows that
 \[
\dim\Com_{\infty,3}(2n-1)=\frac{2^{2n}(2^{2n}-1)|B_{2n}|}{2n}, 
 \]
thus completing the proof.
\end{proof}

\section{Pure homotopy Lie algebras and their mock versions}\label{sec:mockLieInfty}

In this section, we shall discuss the mock (= ``wrong homological degree'') versions of the operad of pure homotopy Lie algebras. Namely, we shall show that unexpected extra relations occur in that operad, forcing it to not be Koszul. We believe it is important to emphasize this difference, since in a range of examples in the literature an algebraic structure referred to as ``strongly homotopy Lie $n$-algebra'' in fact is rather a mock strong homotopy algebra, since its structure operation has homological degree zero. 
In most such cases, existing literature seems to either silently ignore the matter of homological degrees of operations, see e.g. \cite{AI,Dzhu04,Dzhu05,GRR,Kis07,Kis04,LMSZ} or to express a belief that ignoring homological degrees does not change much, e.g. \cite{Bru11} notes the discrepancy but states ``A little care over the sign factors appearing in the constructions would be required. However, this should be very tractable.''. While several available sources mention that the mock $n$-algebras are not really related to homotopy Lie algebras,  see e.g.~\cite{Schr1,Schr2}, there is no clear indication as to what the difference between the two notions is. We believe that the results of this section emphasize such difference.

\smallskip 

We begin with a definition of degree $d$ versions of totally associative commutative algebras and of $n$-Lie algebras. 

\begin{definition}\label{def:tComAndNLie}\leavevmode
\begin{itemize}
\item The operad $\tCom^n_d$ of totally associative commutative $n$-ary algebras with operation in homological degree~$d$ is generated by one element $\mu\in\tCom^n_d(n)$ of degree $d$ which is fully symmetric, that is
 $
\mu\vladimirsdot\sigma = \mu 
 $
for all $\sigma\in S_n$, and
 \[
\mu\circ_i\mu=\mu\circ_j\mu
 \]
for arbitrary $1 \leq i,j \leq n$.
\item  The operad $\Lie^n_d$ of $n$-Lie algebras with operation in homological degree $d$ is generated by one element $\ell\in\Lie^n_d(n)$ of degree $d$ 
which is fully antisymmetric, that is 
 $
\ell\vladimirsdot\sigma=\sgn(\sigma)\ell
 $
for all $\sigma \in S_n$, and 
 \[
\sum_\delta \sgn(\delta)(\ell\circ_1\ell)\vladimirsdot\delta=0,
 \]
where the summation is taken over all $(n,n-1)$-unshuffles $\delta$, that is permutations in $S_{2n-1}$ for which $\delta(1)<\cdots\delta(n)$ and $\delta(n+1)<\cdots<\delta(2n-1)$.
\end{itemize}
\end{definition}

For example, the operad $\tCom^2_0$ is the operad encoding the usual commutative associative algebras, and the operad $\Lie^2_0$ is the operad encoding Lie algebras. 
In general, for $n\geq 2$, the operad $\Lie^n_{n-2}$ encodes $L_\infty$ algebras~\cite{LaMa95} without differential whose only nontrivial operation is in arity $n$.

\medskip 

For the same technical reasons as in~\cite{MaRe15} we introduce
auxiliary operads by means of operadic suspension:
\[
\widetilde{\tCom}^n_d := \calS\hadam \tCom^n_{d+n-1} \ \hbox { and } \ 
\tLie^n_d : = \calS\hadam \Lie^n_{d+n-1}.
\]
By a direct calculation in the endomorphism operad, we obtain the following result. 

\begin{proposition}\label{prop:tComNLieSusp}\leavevmode
\begin{itemize}
\item The operad $\widetilde{\tCom}^n_d$ is generated by a degree $d$ fully antisymmetric
operation $\mu$ of arity $n$ satisfying
 \[
(-1)^{(i-1)d} \mu\circ_i\mu=(-1)^{(j-1)d} \mu\circ_j\mu
 \]
for arbitrary $1 \leq i,j \leq n$.
\item The operad $\tLie^n_d$ is generated by a degree $d$ fully symmetric operation $\ell$ of arity $n$
satisfying 
 \[
\sum_\delta(\ell\circ_1\ell)\vladimirsdot\delta=0,
 \]
where the summation runs over all $(n,n-1)$-unshuffles $\delta$.
\end{itemize}
\end{proposition}

\begin{proposition}\label{prop:KoszulDuals}
One has the isomorphisms
 \[
(\tLie^n_d)^{\ac} \cong (\tCom^n_{1+d})^c\ \text{  and } \  (\Lie^n_d)^{\ac} \cong (\ttCom^n_{1+d})^c.
 \]
(Here the superscript ${}^c$ denotes component-wise dualization of~$\calP$ followed by reversing homological degrees.) 
Equivalently, using the language of Koszul dual operads, 
 \[
(\Lie^n_d)^! \cong \tCom^n_{n-2-d}\ \text{  and } \ (\tLie^n_d)^! \cong \ttCom^n_{n-2-d} .
 \]
\end{proposition}

\begin{proof}
The first isomorphism is established by direct inspection; the others follow from the standard properties of Koszul duality and operadic suspension.
\end{proof}

\begin{theorem}\label{th:MockTCom}
The operad $\tLie^n_d$ is Koszul if and only if $d$ is odd; the operad $\Lie^n_d$ is Koszul if and only if $n$ and $d$ have the same parity.
Equivalently, the operad $\tCom^n_d$ is Koszul if and only if $d$ is even; the operad $\ttCom^n_d$ is Koszul if and only if $n$ and $d$ have
different parities.
\end{theorem}

Note that in the view of Proposition \ref{prop:KoszulDuals} and properties of operadic suspension, it is enough to consider the operad $\tCom^n_d$ for various $n$ and $d$. Let us begin with describing the underlying collection of that operad. 

\begin{proposition}\label{prop:tComDim} For even $d$, the operad $\tCom_d^n$ has a quadratic Gr\"obner basis. Its arity $k$ component $\tCom_d^n$ is one-dimensional for $k\cong 1\pmod{n-1}$ and is equal to zero otherwise.

For odd $d$, the arity $k$ component of the operad $\tCom_d^n$ is one-dimensional for $k=1,n$, or $2n-1$, and is equal to zero otherwise. 
\end{proposition}

\begin{proof}
To prove the first statement, note that for the purpose of Gr\"obner basis computation, only the parity of the homological degree matters, so we may assume $d=0$. In this case, we have $\tCom^n_0=\Com^{[n-1]}$, and the result follows from Proposition~\ref{prop:PBWVeronese}. 

To prove the second statement, we note that a computation similar to that of \cite{DoKh2,MaRe15} shows that in addition to the existing quadratic relations, the Gr\"obner basis contains the element $\mu\circ_{n}(\mu\circ_n\mu)$, and hence there are no normal tree monomials of arity greater than $2k-1$, and the first part of the result follows. \end{proof}

It follows that the operad $\tCom^n_d$ is Koszul for even $d$, proving the ``easy'' part of Theorem \ref{th:MockTCom}. 

For $n=2$ and odd $d$, the operad $\tCom^n_d$ is not Koszul; for
instance it is the case because it has the same Poincar\'e series as
the mock-commutative operad \cite{GetKa}, and the latter is shown to
fail the positivity criterion, see \emph{op.\ cit.}, footnote on
p.~180. It remains to establish non-Koszulness of the operad
$\tCom^n_1$. It turns out that the positivity criterion is
inconclusive for this operad (which we establish below), so another
argument is needed. It turns out that it is more beneficial to pass to
the Koszul dual operad and establish that the operad $\tLie^n_0$ is
not Koszul. We shall accomplish that similarly to \cite[Th.~9]{DMR16}
by showing that the minimal model of $\tLie^n_0$ is not isomorphic to
the cobar construction of the cooperad $(\tCom^n_1)^c$ as it would
have been in the Koszul case.

By Proposition~\ref{prop:tComDim}, the arity $k$ component of the operad $\tCom^n_1$ is one-dimensional for  $k = 1,n$ or $2n-1$,  and is equal to zero otherwise. As a consequence, the cobar construction $\Omega((\tCom^n_1)^c)$ is generated by a fully symmetric generator $\ell$ of arity $n$ and degree $0$,  and one fully symmetric generator $\xi$ of arity $2n-1$ and degree $1$. Elements of the cobar complex can be represented as linear combinations of shuffle tree monomials whose internal vertices are either of arity $n$ or arity $2n-1$. In homological degree $0$, allowed shuffle tree monomials have only $n$-ary internal vertices, and in homological degree $1$ all but one internal vertex are $n$-ary, and the remaining vertex is of arity $2n-1$. In this representation, the differential of the cobar construction is the summation of all possible 
expansions of the vertex of arity $2n-1$ into two $n$-ary vertices.

\begin{figure}
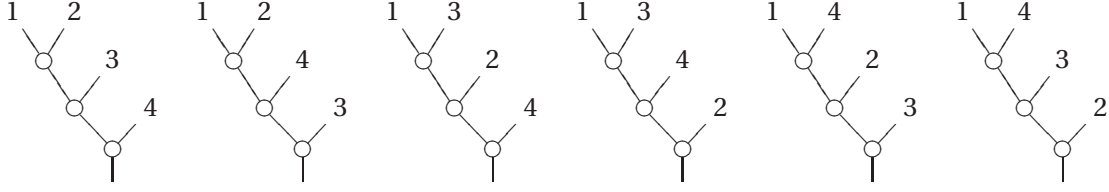

\begin{center}
\llbincomb{}{}{}{1}{2}{3}{4}  \quad 
\llbincomb{}{}{}{1}{2}{4}{3}  \quad
\llbincomb{}{}{}{1}{3}{2}{4}  \quad 
\llbincomb{}{}{}{1}{3}{4}{2}  \quad
\llbincomb{}{}{}{1}{4}{2}{3}  \quad 
\llbincomb{}{}{}{1}{4}{3}{2} 
\end{center}
\caption{\label{o}
The six left combs in $B_1$ for the case of the binary operation $\ell$.
}
\end{figure}

Let us denote by $LC_{(n+1)}$ the set of all $n+1$-fold left comb products of the generator $\ell$ in the cobar complex $\Omega((\tCom^n_1)^c)$. 
The set $LC_{(3)}$ is displayed in \hbox{Figure~\ref{o}}. 

\begin{theorem}
\label{main}
There exist nonzero integers $\epsilon_T \in {\mathbb Z}$ 
given for each shuffle tree $T$ of homological degree $1$ such that for
 \[
\nu := \sum_T \epsilon_T T
 \]
one has
\begin{equation}\label{eq:LeftCombsBoundary}
\partial \nu =  n!\sum_{S \in LC_{(n+1)}} S  \ .
\end{equation}
\end{theorem} 

\begin{proof}
The proof is based on an analogue of Lemma 11 in~\cite{DMR16}, where the only adjustment needed is to replace `planar tree(s)' by
`shuffle tree(s)', and reflect all the pictures in the mirror. In Formula~(7) of that lemma, $B_0 = \emptyset$
while $B_1=LC_{(n+1)}$. This completes the proof.
\end{proof}

\begin{remark}
A notable difference is that, while in~\cite{DMR16} the set $B_1$ consisted of a single planar rooted tree, in our case it is the sum of all shuffle trees with the same underlying planar tree.
\end{remark}

To continue in the proof of the non-Koszulness, we consider the two elements 
\begin{gather}
\alpha_n=\sum_{\sigma \in Sh_1(n^2-1,n-1)}\ell\circ_{1,\sigma} \nu , \\
\beta_n=\sum_{\sigma \in Sh_1(n-1,n^2-1)}\nu\circ_{1,\sigma} \ell .
\end{gather}
Note that 
 \[
\partial \alpha_n=\sum_{\sigma \in Sh_1(n^2-1,n-1)}\ell\circ_{1,\sigma} \partial\nu=\\
=n!\sum_{\sigma \in Sh_1(n^2-1,n-1)}\ell\circ_{1,\sigma}\sum_{T\in LC_{(n+1)}} T= n!\sum_{T\in LC_{(n+2)}} T
 \]
and 
 \[
\partial \beta_n=\sum_{\sigma \in Sh_1(n-1,n^2-1)}\partial\nu\circ_{1,\sigma}\ell =\\
=n!\sum_{\sigma \in Sh_1(n-1,n^2-1)}\sum_{T\in LC_{(n+1)}} T\circ_{1,\sigma}\ell= n!\sum_{T\in LC_{(n+2)}} T
 \]
Thus, $\partial(\alpha_n-\beta_n)=0$. Let us show that $\alpha_n-\beta_n$ is not equal to a boundary of anything in
the cobar complex.

Consider the following shuffle tree monomial
 \[
\omega_n:=\ell\circ_1 \gamma(\xi;\id,\ldots,\id,\ell,\ldots,\ell) ,
 \]
in words, this is a three-level tree obtained by substituting $n-1$ copies of $\ell$ into the last $n-1$ slots of~$\xi$, and then substituting the result in the first slot of~$\ell$. 

Note that $\omega_n$ appears in $\alpha_n$ with a nonzero coefficient, and does not appear in $\beta_n$ (since it is not obtained as shuffle substitution of $\ell$ in the first slot of anything). Thus, it appears with a nonzero coefficient in the cycle $\alpha_n-\beta_n$. However, this monomial cannot appear in the differential of anything: the differential of $\ell$ is zero, and the differential of $\xi$ can only create trees that have an internal edge between two $n$-ary vertices, while $\omega_n$ does not have such edges. This finishes the proof of Theorem~\ref{th:MockTCom}.

\begin{remark}
The last argument should be compared with the discussion of na\"ive Veronese powers in Proposition~\ref{prop:NaiveVeronese}; indeed, this argument is only possible because the na\"ive second Veronese power of the free operad is different from the second Veronese power as defined in this paper.
\end{remark}

\section{The positivity criterion of Koszulness is not decisive for the operad \texorpdfstring{$\tLie^3_0$}{tL3}}\label{sec:inverse}

In this section, we consider the possibility of using the positivity criterion of Koszulness for the operad $\tLie^n_0$.  Since the Koszul dual of this operad is a very simple cooperad $\left(\tCom^n_{1}\right)^c$, it is natural to try to prove non-Koszulness by establishing that the compositional inverse of the Poincar\'e series of the latter cooperad has negative coefficients. This works for $n=2$, but it turns out that already for $n=3$ the inverse series does not have any negative coefficients, which we demonstrate below. The argument is similar to that of \cite{DMR16}.

We first recall a classical result on inversion of power series. To state it, we use, for a  formal power series $F(t)$, the notation $\left[t^k\right]F(t)$ for the coefficient of $t^k$ in $F(t)$, and the notation $F(t)^{\langle -1\rangle}$ for the compositional inverse of $F(t)$ (if that inverse exists).

\begin{proposition}[Lagrange's inversion formula {\cite[Sec.~5.4]{Stan2}}]\label{prop:LIT}
Let $f(t)$ be a formal power series without a constant term and with a nonzero coefficient of $t$.  Then $f(t)$ has a compositional inverse, and
 \[
\left[t^k\right]f(t)^{\langle -1\rangle}=\frac1k \left[u^{k-1}\right]\left(\frac{u}{f(u)}\right)^k .
 \]
\end{proposition}

Let us now prove the main result of this section. Namely, we show that the compositional inverse of the power series $g_{\left(\tCom^3_{1}\right)^c}(t)$ has nonnegative coefficients, and hence the positivity criterion \cite{DMR16} cannot be used to establish the non-Koszulness of the operad $\tLie^n_0$.  

\begin{theorem}
The compositional inverse of the power series 
 \[
g_{\left(\tCom^3_{1}\right)^c}(t)=t-\frac{t^3}{6}+\frac{t^{5}}{120}
 \]
is of the form $t\,h(t^2)$, where $h$ is a power series with positive coefficients. 
\end{theorem}

\begin{proof}
First, let us recall the usual argument explaining the form of the inverse series. By Proposition \ref{prop:LIT}, we have 
 \[
\left[t^k\right](t-\frac{t^3}{6}+\frac{t^{5}}{120})^{\langle -1\rangle}=\frac1k \left[u^{k-1}\right]\left(\frac{u}{u-\frac{u^3}{6}+\frac{u^{5}}{120}}\right)^k=\frac1k \left[u^{k-1}\right]\left(\frac{1}{1-\frac{u^2}{6}+\frac{u^{4}}{120}}\right)^k ,
 \]
and the coefficients on the right vanish unless $k=2n+1$, so the inverse series is of the form $t\,h(t^2)$, where $h$ is some formal power series. 

Let us start the asymptotic analysis of the coefficients of the series $h(t)$. 

\begin{lemma}\label{lm:radius}
The radius of convergence of $h(t)$ is equal to $\frac{16(3+\sqrt{3})}{75}$.
\end{lemma}

\begin{proof}
The radius of convergence of  $\left(t-\frac{t^3}{6}+\frac{t^{5}}{120}\right)^{\langle -1\rangle}$ is equal to the maximal $r$ for which the inverse function of $t-\frac{t^3}{6}+\frac{t^{5}}{120}$ is analytic for the arguments whose modulus is smaller than~$r$. It is obvious that such $r$ is the value of $t-\frac{t^3}{6}+\frac{t^{5}}{120}$ at the modulus of the smallest zero of 
 \[
\left(t-\frac{t^3}{6}+\frac{t^{5}}{120}\right)'=1-\frac{t^2}{2}+\frac{t^{4}}{24} .
 \]
The zeros of the latter are 
 \[
\sqrt{\frac{\frac12\pm\sqrt{\frac14-\frac{4}{24}}}{2\frac1{24}}}=\sqrt{\frac{\frac12\pm\sqrt{\frac1{12}}}{\frac1{12}}}=\sqrt{6\pm2\sqrt{3}},
 \]
so the smallest one is $\rho=\sqrt{6-2\sqrt{3}}$. Thus, the radius of convergence of the inverse series is 
 \[
\rho\left(1-\frac{6-2\sqrt{3}}{6}+\frac{(6-2\sqrt{3})^2}{120}\right)=\rho\frac{6+2\sqrt{3}}{15} .
 \]
Now, as $\left(t-\frac{t^3}{6}+\frac{t^{5}}{120}\right)^{\langle -1\rangle}=t\,h(t^2)$, the radius of convergence of $h(t)$ is equal to 
 \[
\left(\rho\frac{6+2\sqrt{3}}{15}\right)^2=
\frac{(6-2\sqrt{3})(6+2\sqrt{3})^2}{225}=\frac{48(3+\sqrt{3})}{225}=\frac{16(3+\sqrt{3})}{75} .
 \]
\end{proof}

\begin{lemma}
The $n$-th coefficient of $h(t)$ is equal to 
 \[
a_n=\frac{1}{2n+1}\sum_{k=\lfloor n/2\rfloor}^{n} (-1)^{n-k}\binom{2n+k}{k}\binom{k}{n-k}\left(\frac{1}{6}\right)^{2k-n}\left(\frac{1}{120}\right)^{n-k} .
 \]
\end{lemma}

\begin{proof}
Continuing the computation that utilises the Lagrange's inversion formula, we see that the $n$-th coefficient of $h$, or equivalently the coefficient of $t^{2n+1}$ of $(t-\frac{t^3}{6}+\frac{t^{5}}{120})^{\langle -1\rangle}$, is equal to 
 \[
\frac1{2n+1} \left[u^{2n}\right]\left(\frac{1}{1-\frac{u^2}{6}+\frac{u^{4}}{120}}\right)^{2n+1}=\frac1{2n+1} \left[v^{n}\right]\left(\frac{1}{1-\frac{v}{6}+\frac{v^2}{120}}\right)^{2n+1} ,
 \] 
and expanding the latter using the binomial theorem, we get
 \[
\left(\frac{1}{1-\frac{v}{6}+\frac{v^2}{120}}\right)^{2n+1}=\sum_{k\ge 0}\binom{2n+k}{k}\left(\frac{v}{6}-\frac{v^2}{120}\right)^k ,
 \]
therefore the coefficient of $t^{2n+1}$ is given by  
 \[
\frac{1}{2n+1}\sum_{k\ge0}\sum_{2i+k-i=n} \binom{2n+k}{k}\binom{k}{i}\left(\frac{1}{6}\right)^{k-i}\left(-\frac{1}{120}\right)^i .
 \] 
Clearly, $i=n-k$, so this simplifies to
 \[
a_n=\frac{1}{2n+1}\sum_{k=\lfloor n/2\rfloor}^{n}(-1)^{n-k}\binom{2n+k}{k}\binom{k}{n-k}\left(\frac{1}{6}\right)^{2k-n}\left(\frac{1}{120}\right)^{n-k} ,
 \] 
as required.
\end{proof}

The expression $a_n$ is given by the formula which is a sum of ``hypergeometric'' terms, we see that Zeilberger's algorithm \cite[Ch.~6]{AeqB} applies. We used the interface to it provided by the \texttt{sumrecursion} function of \textsf{Maple}; this function implements the Koepf's version of Zeilberger's algorithm \cite[Ch.~7]{Koepf}. This function instantly informs us that the sequence $\{a_n\}$ is a solution to a rather remarkable three term finite difference equation 
\begin{equation}\label{eq:3term}
s_0(n)x_n-s_1(n)x_{n-1}+s_2(n)x_{n-2}=0 , 
\end{equation}
where
\begin{align*}
s_0(n)&=128n(n-1)(2n+1)(2n-1)(5n-6) , \\
s_1(n)&=80(n-1)(2n-1)(5n-1)(15n^2-30n+14) ,\\
s_2(n)&=3(5n-1)(5n-4)(5n-6)(5n-7)(5n-8) .
\end{align*}

The polynomials $s_i(n)$ are of the same degree $5$, and so our equation is of the type considered by Poincar\'e in \cite{Po}. Namely, in \cite[\S2]{Po} linear finite difference equations of order $k$
 \[
s_0(n)x_n+s_1(n)x_{n-1}+\cdots+s_k(n)x_{n-k}=0 
 \]
are considered, with the additional assumption that $s_0(n)$, \ldots, $s_k(n)$ are polynomials of the same degree $d$. To such an equation, one associates its characteristic polynomial
 \[
\chi(t)=\alpha_0t^k+\alpha_1t^{k-1}+\cdots+\alpha_k=0 ,
 \]
where $\alpha_i$ is the coefficient of $t^d$ in $s_i(n)$. If the absolute values of the complex roots of $\chi(t)$ are pairwise distinct, then for any solution $\{a_n\}$ to our equation, the limit 
 \[
\lim_{n\to\infty}\frac{a_n}{a_{n-1}}
 \]
exists and is equal to one of the roots of $\chi(t)$. Usually, that root will be the one which is maximal in absolute value. The particular case when the root is the minimal in absolute value is the hardest both for computations and for the qualitative analysis of the asymptotic behaviour, since in this case the corresponding solution is unique up to proportionality, and so the situation is not stable under small perturbations.  In our case the polynomial $\chi(t)$ is
 \[
2560t^2-12000t+9375 ,
 \]
and its roots are 
 \[
\lambda_-=\frac{25(3-\sqrt{3})}{32}\approx 0.9905853066 \quad \text{and}\quad \lambda_+=\frac{25(3+\sqrt{3})}{32}\approx 3.696914693, 
 \]
so Poincar\'e theorem applies. In fact, $\lambda_-$ is immediately seen to be the inverse of the radius of convergence of $h(t)$. By the usual ratio formula for the radius of convergence, we see that 
 $
\lim_{n\to\infty}\frac{a_{n}}{a_{n-1}}=\lambda_- .
 $

Let us consider the auxiliary sequence $\{b_n\}$ satisfying the same finite difference equation \eqref{eq:3term} and the initial conditions $b_0=1$, $b_1=1$. 

\begin{lemma}
All terms of the sequence $\{b_n\}$ are positive for $n>0$, and we have 
 \[
\lim_{n\to\infty}\frac{b_n}{b_{n-1}}=\lambda_+.
 \]
\end{lemma}

\begin{proof}
First, let us show that for all $n\ge0$ we have $\frac{b_n}{b_{n-1}}\ge 1$. This is easy to see by induction on $n$. First, for $n=1$, the statement is obvious. Next, if we suppose that it is true for all values less than the given $n$, we have
 \[
\frac{b_n}{b_{n-1}}=\frac{s_1(n)}{s_0(n)}-\frac{s_2(n)b_{n-2}}{s_0(n)b_{n-1}}\ge \frac{s_1(n)}{s_0(n)}-\frac{s_2(n)}{s_0(n)},
 \]
and so it suffices to show that 
 \[
\frac{s_1(n)}{s_0(n)}-\frac{s_2(n)}{s_0(n)}\ge1 .
 \]
The roots of the polynomial $s_0(n)$ are $0,1,-\frac12,\frac12,\frac65$, so this polynomial assumes positive values in the given range. Thus, the above inequality is equivalent to 
 \[
0>s_0(n)-s_1(n)+s_2(n)=-65n^5 - 9982n^4 + 36457n^3 - 45402n^2 + 21832n - 2912 .
 \]
Using computer algebra software, we find that the latter polynomial has the largest root approximately equal to $1.403$, so for all $n\ge2$ it assumes negative values, and step of induction is proved. Also, by Poincar\'e Theorem, the limit of the ratio $\frac{b_n}{b_{n-1}}$ as $n\to\infty$ is equal to either $\lambda_-$ or $\lambda_+$. However, 
$1>\lambda_-$, so the inequality $\frac{b_n}{b_{n-1}}\ge 1$ shows that the first of the two alternatives is impossible. Hence, the limiting value is $\lambda_+$.  
\end{proof}

Our results thus far imply  that 
 $
\lim\limits_{n\to\infty}\frac{a_n}{b_n}=0,
 $ 
as
 \[
\frac{a_{n+1}}{b_{n+1}}=\frac{a_n}{b_n}\frac{\frac{a_{n+1}}{a_n}}{\frac{b_{n+1}}{b_n}},
 \]
and so $\frac{a_{n+1}}{b_{n+1}}$ is a multiple of $\frac{a_n}{b_n}$ by a factor close to $\frac{\lambda_-}{\lambda_+}<1$ for large~$n$, and thus our sequence can be bounded from above in absolute value by a geometric sequence with a zero limit.

Now it is easy to complete the proof. We note that 
\begin{multline*}
\frac{a_n}{b_n}-\frac{a_{n-1}}{b_{n-1}}=\frac{s_1(n)a_{n-1}-s_2(n)a_{n-2}}{s_1(n)b_{n-1}-s_2(n)b_{n-2}}-\frac{a_{n-1}}{b_{n-1}}=\\
\frac{(s_1(n)a_{n-1}-s_2(n)a_{n-2})b_{n-1}-(s_1(n)b_{n-1}-s_2(n)b_{n-2})a_{n-1}}{(s_1(n)b_{n-1}-s_2(n)b_{n-2})b_{n-1}}=\\
\frac{s_2(n)(a_{n-1}b_{n-2}-a_{n-2}b_{n-1})}{s_0(n)b_nb_{n-1}}=
\frac{s_2(n)b_{n-2}}{s_0(n)b_n}\left(\frac{a_{n-1}}{b_{n-1}}-\frac{a_{n-2}}{b_{n-2}}\right)
\end{multline*}
The roots of the polynomial $s_2(n)$ are $\frac15,\frac45,\frac65,\frac75,\frac85$, so for $n\ge 2$ the sign of $\frac{a_n}{b_n}-\frac{a_{n-1}}{b_{n-1}}$ is the same as the sign of $\frac{a_{n-1}}{b_{n-1}}-\frac{a_{n-2}}{b_{n-2}}$, and hence the same as the sign of 
\[\frac{a_1}{b_1}-\frac{a_0}{b_0}=-\frac56<0 .\] Thus, $\left\{\frac{a_n}{b_n}\right\}$ is a strictly decreasing sequence.  For a decreasing sequence with the limit zero, all terms must be positive, and hence $a_n$ is positive for all $n>0$. 
\end{proof}

\bibliographystyle{amsplain}
\providecommand{\bysame}{\leavevmode\hbox to3em{\hrulefill}\thinspace}

\end{document}